\documentclass{article}

\usepackage{amssymb,amsfonts,amsthm,amsmath,dsfont} 
\usepackage{fullpage}
\usepackage{graphicx}
\usepackage[usenames,dvipsnames,svgnames,table]{xcolor}
\usepackage[utf8]{inputenc}
\usepackage{esint} 
\usepackage[colorlinks=true, pdfstartview=FitV,linkcolor=ForestGreen,citecolor=ForestGreen, urlcolor=blue]{hyperref}
\usepackage[shortlabels]{enumitem}

\newcommand{\R}{\mathbb{R}}

\newcommand{\eps}{\varepsilon}

\newcommand{\cR}{\mathcal{R}}
\newcommand{\cQ}{\mathcal{Q}}

\newcommand{\cG}{\mathcal{G}}

\newcommand{\ind}{\mathds{1}}
\newcommand{\dt}{\, \mathrm{d} t}
\newcommand{\dx}{\, \mathrm{d} x}

\newcommand{\Ccfl}{C_{\mathrm{CFL}}}

\DeclareMathOperator{\sgn}{sgn}
\DeclareMathOperator{\BV}{BV}

\newtheorem{theo}{Theorem}[section]
\newtheorem{lem}[theo]{Lemma}
\newtheorem{pro}[theo]{Proposition}

\newtheorem{defi}[theo]{Definition}
\theoremstyle{remark}
\newtheorem{rem}[theo]{Remark}

\numberwithin{equation}{section}

\begin{document}

\title{\bf Germs for scalar conservation laws: \\
  the Hamilton-Jacobi equation point of view}
\author{Nicolas Forcadel, Cyril Imbert et Régis Monneau}
\date{\today \\[2ex] \textit{Version 6}}

\maketitle

\begin{abstract}
  We prove that the entropy solution to a scalar conservation law posed on the real line with a  flux that is discontinuous
  at one point (in the space variable) coincides with the derivative of the solution to a Hamilton-Jacobi (HJ) equation whose Hamiltonian is discontinuous.
  Flux functions (Hamiltonians) are not assumed to be convex in the state (gradient) variable. The proof consists in proving the
  convergence of two numerical schemes.  We rely on the theory developed
  by B.~Andreianov, K.~H.~Karlsen and N.~H.~Risebro (\textit{Arch. Ration. Mech. Anal.}, 2011) for such scalar conservation laws
  and on the viscosity solution theory developed by the authors (\textit{arxiv}, 2023) for the corresponding HJ equation.
  This study allows us to
  characterise certain germs introduced in the AKR theory (namely maximal and complete ones) and relaxation operators introduced in the viscosity solution framework.
\end{abstract}

\setcounter{tocdepth}{1}
\tableofcontents


\bigskip

\section{Introduction}

  We are interested in scalar conservation laws (SCL) with discontinuous
  flux posed on the real line. The discontinuity arises at the origin in space variable, with a flux on the left and a possible different flux on the right.
  As far as entropy solutions for such equations are concerned,
  we adopt here the (AKR) point of view of B.~Andreianov, K.~H.~Karlsen and
  N.~H.~Risebro \cite{MR2807133}.  The
  condition imposed to the entropy solution at the discontinuity
  concerns its strong traces from the right and the left. They are
  imposed to belong to a set that is called a \emph{germ}.
  Following the AKR theory, uniqueness of the solution is known for maximal germs (in the sense
  of inclusion), and existence of the solution is known  for complete germs (for which the Riemann problem can be solved).

   We can assert at least formally that the solution of a SCL
  is the derivative of the solution of a Hamilton-Jacobi (HJ) equation whose Hamiltonian
  coincides with the flux function of the conservation law. In our
  framework, this   Hamiltonian is thus discontinuous (in the spatial variable).
  In the viscosity solution theory developed for the corresponding HJ equations,  conditions imposed at the discontinuity 
  are in correspondance with a family of monotone nonlinearities. The relaxation of such nonlinearities creates naturally some $\mathcal G$-Godunov fluxes for a certain germ $\mathcal G$.

  When flux functions / Hamiltonians are convex in
  the state / gradient variable, it was recently proved by
  P.~Cardaliaguet, T.~Girard and the first and third authors
  \cite{cardaliaguet2023conservation} that the spatial derivative of the viscosity
  solution is the entropy solution for a germ associated with the
  boundary nonlinearity. We prove that this result still holds true
  for general Hamiltonians, not necessarily convex, but coercive. The problem is
  significantly more difficult because conditions at the discontinuity
  (for both equations) are much richer.
  Indeed, the proof of the
  result consists in approximating solutions of the two equations by
  numerical schemes and in proving their convergence. The difficulty
  lies in identifying the germ selected by the numerical scheme
  associated with the SCL.

  It is one of the main contributions of this work to show that the desired condition at the
  discontinuity of
  the limit of the SCL numerical scheme  is necessarily relaxed. 
  Such a phenomena was  exhibited by  the authors at the level
  of  HJ equations \cite{MR3621434,forcadel2023nonconvex}. Its understanding is used to address
  the relaxation of the condition at the discontinuity for the SCL.

 Another contribution of this work is to show that the germs selected by this approximation procedure
coincide with maximal  (in the sense of inclusion) and complete  (for which the Riemann
problem can be solved) ones.
   The derivation of this formula is based on a Hamilton-Jacobi point of view on the problem.

\subsection{Scalar conservation laws and Hamilton-Jacobi equations}

In this article, we consider  a scalar conservation law of the form
\begin{equation}\label{eq:scl}
\begin{cases}
v_t + H_L (v)_x=0,  & t>0,x<0, \\
v_t + H_R (v)_x = 0, & t>0, x>0, \\
(v(t,0-),v(t,0+)) \in \mathcal{G}, & t>0   
\end{cases}
\end{equation} 
where $H_L,H_R \colon \R \to \R$ satisfy,
\begin{equation}
  \label{e:assum-H}
  \begin{cases}
  H_\alpha \text{ are $\mathfrak{L}_\alpha$-Lipschitz continuous},\\  
 H_\alpha \text{ is not constant on any open interval},\\
 H_\alpha (p_\alpha) \to + \infty \quad \text{ as } |p_\alpha| \to + \infty 
\end{cases}
\end{equation}
with $\alpha \in \{L,R\}$.

The condition at $x=0$ for entropy solutions that we will work with in this article
was  introduced  by B. Andreianov, K.~H.~Karlsen and N.~H.~Risebro in \cite{MR2807133}.
It is necessary to supplement the equation with a condition at $x=0$ (even if $H_L = H_R$).
This condition amounts to impose that the couple of traces $(v(t,0+),v(t,0-))$ lie in a given set  $\mathcal{G}$, called the \emph{germ}. 

In this work, we make precise the link between such entropy solutions of \eqref{eq:scl} associated to a germ $\mathcal{G}$ and
viscosity solutions of the following Hamilton-Jacobi equation,
\begin{equation}\label{eq:hj}
\begin{cases}
  u_t + H_L(u_x)=0, & t>0,x<0, \\
  u_t + H_R(u_x)=0, & t>0, x>0, \\
u_t + F_0(u_x (t,0-), u_x (t,0+))=0, & t>0, x=0.
\end{cases}
\end{equation}
We assume that the function $F_0$ satisfies
\begin{equation}
  \label{e:assum-F}
  \begin{cases}
F_0 \text{ is $\mathfrak{L}_0$-Lipschitz continuous,} \\
   p_L \mapsto F_0(p_L,p_R) \text{ is non-decreasing,} \\
   p_R \mapsto F_0(p_L,p_R) \text{ is non-increasing,}\\
  F_0 (p) \to +\infty \text{ as } (p_L)_+ + (p_R)_- \to +\infty
\end{cases}
\end{equation}
with $(p_\alpha)_- = \max (0,-p_\alpha)$. 

Both equations are supplemented with an initial condition,
\begin{eqnarray}
 \label{eq:ic-hj} u &=u_0 \quad \text{ in }\quad \left\{0\right\}\times \R , \\
                      \label{eq:ic-scl}    v &= v_0 \quad \text{ in }\quad \left\{0\right\}\times \R
\end{eqnarray}
with  $v_0 = (u_0)_x \in L^\infty \cap \BV (\R)$.\\
In short, we say that {$v$} is a $\mathcal G$-entropy solution to \eqref{eq:scl}, while {$u$} is a $F_0$-viscosity solution to \eqref{eq:hj}.

\subsection{Numerical schemes}

We now describe the numerical scheme used to solve the Hamilton-Jacobi equation \eqref{eq:hj}. Given a time step $\Delta t>0$ and a space step $\Delta x>0$, we consider the discrete time $t_n=n\Delta t$ for $n\in \mathbb N$ and the discrete point $x_{j}=j\Delta x$ for $j\in \mathbb Z$. We denote by $u^n_j$ the numerical approximation of $u(t_n,x_j)$. In order to discretize \eqref{eq:hj}, we will use a Godunov approximation. More precisely, we introduce the following Godunov numerical Hamiltonians, for $\alpha=L,R$
\begin{equation}\label{eq::C16}
g^{H_\alpha}(p^-,p^+)=\begin{cases}
\displaystyle \min_{p\in [p^-,p^+]}H_\alpha(p) & \text{ if }  p^-\le p^+,\\
\displaystyle \max_{p\in [p^+,p^-]}H_\alpha(p)& \text{ if } p^+\le p^-.
\end{cases}
\end{equation}
We remark that the functions $g^{H_\alpha}$ are  non-decreasing in the first variable and non-increasing in the second one. Moreover, $g^{H_\alpha}(p,p)=H_\alpha(p)$ for $\alpha=R,L$.
Given $n \ge 1$, we define for all $j\in \mathbb Z$, 
\begin{equation}\label{eq:DefvDelta}
v^n_{j+\frac 12}=\frac{u^n_{j+1}-u^n_{j}}{\Delta x}.
\end{equation}
The numerical scheme is then given by
\begin{equation}\label{eq:scheme-HJ}
\begin{cases}
\displaystyle{\frac {u^{n+1}_{j}-u^n_{j}}{\Delta t}+g^{H_L}  \left( v^n_{j-\frac 12},v^n_{j+\frac 12} \right)=0}& \text{ for } j\le -1,\\[2ex]
\displaystyle{\frac {u^{n+1}_{j}-u^n_{j}}{\Delta t}+g^{H_R}  \left( v^n_{j-\frac 12},v^n_{j+\frac 12} \right)=0}&\text{ for } j\ge 1,\\[2ex]
\displaystyle{\frac {u^{n+1}_{j}-u^n_{j}}{\Delta t}+F_0  \left(v^n_{j-\frac 12}, v^n_{j+\frac 12} \right)=0}&\text{ for } j=0
\end{cases}
\end{equation}
completed with the initial condition
\[u^0_j=u_0(j\Delta x)\quad {\rm for}\; j\in \mathbb Z.\]

Given $u_0$, we consider $v_0 :=  (u_0)_x$ and its discretized version
\begin{equation}\label{e:v0-discrete}
  v^0_{j+\frac12}= \frac{u_{j+1}^0 - u_j^0}{\Delta x} = \frac{u_0(x_{j+1}) - u_0(x_{j}) }{\Delta x} = \frac{1}{\Delta x}\int_{x_{j}}^{x_{j+1}} v_0(y) \,\mathrm{d} y.
\end{equation}
The scheme for \eqref{eq:scl} is directly derived from the scheme \eqref{eq:scheme-HJ}. Indeed, recalling the definition of $v^n_{j+1/2}$ in \eqref{eq:DefvDelta}, we can write
\begin{equation} \label{eq:scheme-SCL}
  v^{n+1}_{j+\frac 12}=v^n_{j+\frac 12}-\frac{\Delta t}{\Delta x}
  \left(f_{j+1}(v^n_{j+\frac 12},v^n_{j+\frac 32})-f_j(v^n_{j-\frac 12}, v^n_{j+\frac 12})\right)
  \end{equation}
where
\begin{equation}
  \label{e:def-fj}
  f_j (a,b) = \begin{cases}
                g^{H_L} (a,b) & \text{ if } j \le -1 \\
                g^{H_R}(a,b) & \text{ if } j \ge 1 \\
                F_0 (a,b) & \text{ if } j =0.
              \end{cases}
\end{equation}
             
It is convenient to introduce the functions $\mathcal{F}_j$ from the right-hand side of the above scheme: for any $n,j$, we have
\begin{equation}
  \label{e:monotone-scheme}
  v^{n+1}_{j+\frac12} = \mathcal{F}_j(v^n_{j-\frac12},v^n_{j+\frac12},v^n_{j+\frac32}).
\end{equation}
We shed light on the fact that $\mathcal{F}_j(v^n_{j-\frac12},v^n_{j+\frac12},v^n_{j+\frac32})$ is monotone with respect to
$v^n_{j-\frac12}$ and $v^n_{j+\frac32}$ only if $F_0(a,b)$ satisfies the monotonicity properties
given in \eqref{e:assum-F}. A condition of CFL type ensures that it is monotone with respect to $v^n_{j+\frac12}$, 
\begin{equation}
  \label{e:CFL}
\frac{\Delta t}{\Delta x} \le  (2 \mathfrak{L})^{-1} =: \Ccfl
\end{equation}
with
\[\mathfrak{L} = \max_{\alpha \in \{0,L,R\}} \mathfrak{L}_\alpha.\]
We recall that for $\alpha \in \{L,R\}$, $\mathfrak{L}_\alpha$ denotes the Lipschitz constant of $H_\alpha$  (and $g^{H_\alpha}$ in both variables) and $\mathfrak{L}_0$ denotes the Lipschitz constant of $F_0$ (in both variables). 

It is convenient to define a function $u_\Delta$ in continuous variables $(t,x) \in (0,+\infty) \times \R$ by linear interpolation: for $t>0$ and $x \in \R$,
\begin{equation}
  \label{e:u-delta}
  u_{\Delta}(t,x) := \sum_{n \in \mathbb{N}} \ind_{[t_n, t_{n+1})}(t) \ind_{[x_j, x_{j+1})}(x) \left[ u_j^n + \frac{u_{j+1}^{n} - u_j^n}{\Delta x} (x-x_{j}) \right].
\end{equation}
where  $\Delta$ stands for $(\Delta t,\Delta x)$. 
Similarly, we define $v_\Delta$ on $(0,+\infty) \times \R$ by,
\begin{equation}\label{e:v-delta}
   v_{\Delta} := \sum_{n\in \mathbb{N}} \sum_{j\in\mathbb{Z}} v_{j+\frac12}^n \ind_{[t_n, t_{n+1})\times [x_j, x_{j+1})} .
\end{equation}

\subsection{Main results}

The first main result of this article asserts that the spatial derivative of the
viscosity solution of the Hamilton-Jacobi (HJ) equation coincides with the entropy solution
of the corresponding scalar conservation law (SCL).
More precisely, if at the HJ level, the junction condition is encoded by the nonlinearity $F_0$, then at the SCL level, the associated germ is $\mathcal G_{\mathcal R F_0}$ that is represented by the ``relaxed'' nonlinearity $\cR F_0$.
 The definition of the relaxation operator is recalled
in the next section. 
\begin{theo}[Link between HJ and SCL] \label{t:link}
Let $u_0$ be Lipschitz continuous and $v_0 =(u_0)_x$ be of bounded variation in $\R$. 
Let $u \colon (0,+\infty) \times \R \to \R$ be the unique $F_0$-viscosity solution of
\eqref{eq:hj}-\eqref{eq:ic-hj} and $v$ be the unique $\mathcal G$-entropy solution of \eqref{eq:scl},\eqref{eq:ic-scl}
with the germ $\mathcal G:=\mathcal G_{F}$ defined by
\begin{equation}\label{eq::C12}
\cG_{F} = \{ (p_L,p_R) \in \R^2: H_L (p_L) = H_R (p_R) = F (p_L,p_R) \}
\end{equation}
  where $F:=\cR F_0$ (defined in \eqref{eq::C1}) is the relaxed junction condition at $x=0$.
  Then $v=u_x$ in $L^\infty$. 
\end{theo}
  \begin{rem}
    Since we want to relate solutions for scalar conservation laws with the ones for Hamilton-Jacobi
equations, the functions $H_L$ and $H_R$ are both flux functions and Hamiltonians. We make assumptions
on $H_L$ and $H_R$ that are convenient to work with for both equations. More precisely, the continuity and
coercivity of the $H_\alpha$'s is used in the HJ framework while the Lipschitz continuity and
the non-degeneracy (no open interval on which $H_\alpha$ is constant) is used at the level of the conservation law. 
\end{rem}
\begin{rem}
  The assumptions we make on $u_0$ are also necessary either for solving the Hamilton-Jacobi equation
  of for solving the scalar conservation law (with the initial data $v_0 =(u_0)_x$). It is convenient
  to assume that $u_0$ is globally Lipschitz so that viscosity solutions are also globally Lipschitz
  and entropy solutions are essentially bounded. The BV assumption on $v_0$ is convenient to prove
  the convergence of the numerical scheme associated with the scalar conservation law. 
\end{rem}

This theorem derives from the convergence of {the} two numerical schemes associated
with each equation. 
\begin{theo}[Convergence of the numerical scheme for SCL] \label{t:num-scheme-scl}
Let $\Delta t$ and $\Delta x$ satisfy the CFL condition~\eqref{e:CFL}. 
  The function $v_\Delta$ weakly converges in $L^1_{\mathrm{loc}} ((0,+\infty) \times \R)$
  as $\Delta x \to 0$  towards
  the unique $\mathcal G_{F}$-entropy solution $v$ of \eqref{eq:scl},\eqref{eq:ic-scl}, with $F:=\mathcal R F_0$ defined in \eqref{eq::C1}.
\end{theo}
{\begin{rem}
Let us point out that even if we put the desired junction condition $F_0$ in the numerical scheme, we recover at the limit $\Delta x\to0$ the relaxed junction condition $\mathcal G_{\mathcal R F_0}$.
\end{rem}}
\begin{theo}[Convergence of the numerical scheme for HJ] \label{t:num-scheme-hj}
Let $\Delta t$ and $\Delta x$ satisfy the CFL condition~\eqref{e:CFL}. 
  The function $u_\Delta$ converges locally uniformly 
  as $\Delta x \to 0$  towards
  the unique $F_0$-viscosity solution $u$ of \eqref{eq:hj},\eqref{eq:ic-hj}. 
\end{theo}
{
\begin{rem}
Let us note that even if the results in Theorem \ref{t:num-scheme-scl} and Theorem \ref{t:num-scheme-hj} seems different since, in the last one, we recover a $F_0$-viscosity solution and not a $\mathcal R F_0$-viscosity solution, this is not the case. Indeed, we know from \cite{FM} that $u$ is a $F_0$-viscosity solution if and only if u is $\mathcal R F_0$-viscosity solution.
\end{rem}
}

\paragraph{Open question.} {Theorem \ref{t:num-scheme-scl} is proved in the case where the junction condition is not relaxed (\textit{i.e.} $F_0\ne \mathcal R F_0$), using the result of Theorem \ref{t:num-scheme-hj}.}
Is it possible to show the result of Theorem \ref{t:num-scheme-scl} directly at the level of scalar conservation laws, without using the HJ framework? Moreover, is it possible to get an error estimate on the difference $|v_\Delta-v|_{L^1}$?
\medskip

The next theorem concerns germs. Roughly speaking, a germ is a collection of  admissible
strong traces for entropy solutions of the conservation law~\eqref{eq:scl}. 
It is maximal if it is not contained in a bigger germ and complete if the Riemman
problem can be solved for all admissible strong traces. Precise definitions are given
in the next section. We prove that a germ is maximal and complete
if and only if it is  represented by a ``self-relaxed nonlinearity''. By doing so, we enrich the AKR theory from \cite{MR2807133}.
\begin{theo}[Classification of maximal and complete germs]\label{t:germ}
  Let $\cG$ be a germ for \eqref{eq:scl}. The following properties are equivalent.
  \begin{enumerate}[(i)]
    \item \label{0:germ} there exists a function $F_0$  satisfying \eqref{e:assum-F} such that $\cG = \cG_{\cR F_0}$.
    \item \label{i:germ} The germ $\cG$ is maximal and complete.
    \end{enumerate}
  \end{theo}

\subsection{Brief review of literature}

\paragraph{Scalar conservation laws with discontinuous flux.}
Contributions to the study of scalar conservation laws with discontinuous flux are numerous.
We can cite for instance the work by F.~Bachmann and J.~Vovelle \cite{zbMATH05025020} where
the flux function is only assumed to be $C^1$ and their uniqueness proof do not require the
existence of strong traces. The reader is referred to the introduction of this work for
the reference containing the model or for previous mathematical contributions
under stronger assumptions. The book by M.~Garavello and B.~Piccoli \cite{zbMATH05130020} was
also influential: the network geometrical setting involves to consider flux functions with
discontinuities at edges. B. Andreianov, K.~H.~Karlsen and N.~H.~Risebro in \cite{MR2807133}
developed a general theory of semi-groups of entropy solutions associated with a scalar
conservation law on the real line with a discontinuity. In particular, they shed light on
the fact that several conditions can be imposed at the discontinuity and they can be characterized
in terms of a set that they refer to as a \emph{germ}. The interested reader is refer to recent
survey articles such as \cite{zbMATH06720857,zbMATH06541454} for more references about this
line of research. {We also refer to \cite{Musch22,Fjord22} for the extension to junctions.}

\paragraph{Hamilton-Jacobi equations with discontinuous Hamiltonians.}
The study of Hamilton-Jacobi equations with discontinuous (in $x$) and convex (in $p=u_x$) Hamiltonians developed
with the study of these equations on networks \cite{zbMATH06189399,zbMATH06144570}.
These first contributions are closely related to optimal control of trajectories in a two-domains framework \cite{zbMATH06198069}. Many contributions followed these three articles and the reader is referred to the book by G.~Barles and E.~Chasseigne \cite{zbMATH07814183} for an up-to-date state of the art, including original contributions to the topic. 

\paragraph{Boundary conditions.}
The fact that the boundary condition can be lost by solutions of first order equations is classical.
Beyond transport effect and the fact that characteristics can exit the domain, an important contribution
to this subject is the work by C.~Bardos, A.-Y.~Le Roux and J.~C.~Nédélec about the Dirichlet
problem for a scalar conservation law \cite{zbMATH03651023}. They gave a weak formulation of the problem
by passing to the limit in the viscous approximation. 
As far as Hamilton-Jacobi equations are concerned, the boundary condition that is effectively
obtained when imposing one that is compatible with the maximum principle were first described
for convex Hamiltonians \cite{MR3621434}. In this
case, any relaxed boundary condition is characterized by a real number, amounting for the limitation
of the ``flux'' at the discontinuity (see also \cite{MR3709301}). The non-convex case is much richer,
the class of relaxed boundary conditions is much larger. It was first studied by J.~Guerand~\cite{MR3695961}
and recently revisited by the authors \cite{forcadel2023nonconvex}{(see also \cite{FM} for the case of junctions)}. As observed in \cite{forcadel2023nonconvex},
it is remarkable that the relaxation
operator can be described in terms of Godunov fluxes appearing in the BLN condition. 

\paragraph{Comparison principles for Hamilton-Jacobi equations.}
In order to prove that the numerical scheme converges, it is important to prove a comparison principle.
The first results in this direction are contained in \cite{MR3729588}. More recently, the authors developed
a new strategy to prove such a strong uniqueness result \cite{forcadel2023coercive,forcadel2024twin}.
See also \cite{FM} for the case of several branches.

\paragraph{Numerical schemes.}
A general theorem for numerical schemes for Hamilton-Jacobi equation (and more generally second order
nonlinear parabolic equations) \cite{MR1115933} asserts that they converge as soon as they are monotone,
stable and consistent. 
A numerical scheme for convex Hamilton-Jacobi equations on a junction was first studied in \cite{MR3311457}.
It was motivated by applications to traffic flow. An error estimate was obtained in \cite{MR3962893}.

As far as scalar conservation laws are concerned, the convex case of the problem adressed in the present
article was recently treated in \cite{cardaliaguet2023conservation}. More generally, for examples of numerical schemes on junctions with $N\ge 1$ branches, we refer the reader to \cite{Musch22,Fjord22}.


\subsection{Organisation of the article}

The article is organised as follows. In Section~\ref{s:prelim}, definitions of germs, solutions and relaxation operators are recalled. We also show that germs $\mathcal G_{\mathcal R F_0}$ are maximal.
In Section~\ref{s:criterion}, we recall a criterion for checking that a function is a viscosity solution of the Hamilton-Jacobi equation~\eqref{eq:hj} and we establish a similar criterion for entropy solutions of the scalar conservation law~\eqref{eq:scl}. The numerical approximation of Hamilton-Jacobi equations is addressed in Section \ref{s:hj-scheme}, where the  convergence of the numerical solution is done (proof of Theorem~\ref{t:num-scheme-hj}).
The numerical approximation of the scalar conservation law
is studied in Section~\ref{s:scl-scheme}, in particular, Theorem~\ref{t:num-scheme-scl} is proved.
 This section also contains the (short) proof of Theorem~\ref{t:link}. The final (short) Section~\ref{s:germs} is devoted to the proof of Theorem~\ref{t:germ}.

\paragraph{Notation.}
For two real numbers $a,b$, the maximum of $a$ and $b$ is denoted by $a \vee b$ while the minimum is denoted by $a \wedge b$. The positive part of $a$ is defined by $a \vee 0$ and is denoted by $a_+$. The negative part $a_-$ of $a$ is defined by $\max(0,-a)$. 

The Lipschitz constant of the Hamiltonian/flux function $H_\alpha$ for $\alpha \in \{R,L\}$ is denoted by $\mathfrak{L}_\alpha$ while $\mathfrak{L}_0$ denotes the Lipschitz constant of $F_0$ (in all variables).

\section{Germs, entropy solutions,  viscosity solutions}
\label{s:prelim}

\subsection{Germs}

In their study of scalar conservation laws with discontinous flux, B.~Andreianov, K.~H.~Karlsen and N.~H.~Riserbro introduced the notion of germs. 
In order to recall their definition, we first recall the definition of the entropy flux functions $q_L$ and $q_R$ associated with the fluxes (or nonlinearities) $H_L$ and $H_R$. 
\begin{equation}\label{e:qalpha}
\forall a, b \in \R, \quad    q_\alpha (a,b) = \sgn (a-b) (H_\alpha (a) - H_\alpha (b)), \qquad \alpha \in \{L,R\}.
\end{equation}
The definition of germs relies on the notion of dissipation. We recall that  it is defined as follows,
\[
  \forall P= (p_L,p_R),P'=(p_L',p_R') \in \R^2, \quad D (P,P') = q_L (p_L,p_L') - q_R (p_R,p_R').
\]

\begin{defi}[Germs]\label{defi:germ}
A set $\mathcal{G} \subset \R^2$ is a \emph{germ} for \eqref{eq:scl} if it satisfies
\begin{itemize}
  \item the \emph{Rankine-Hugoniot condition}: for all $(p_L,p_R) \in \mathcal{G}$, we have $H_L (p_L) = H_R (p_R)$. 
  \item the \emph{dissipation condition}: for all $P, P' \in \mathcal{G}$, we have $ D(P,P') \ge 0$. 
  \end{itemize}
\end{defi}
\begin{rem} \label{r:def-germ}
  In \cite{MR2807133}, a set $\mathcal{G}$ is called an admissible germ if it  only satisfies the Rankine-Hugoniot condition and is called $L^1$-dissipative if it also satisfies the dissipation condition. We will simply call them germs.  
\end{rem}

Important examples of germs are the ones coming from a junction function $F{_0}$,
\begin{equation}\label{eq::C23}
\mathcal{G}_{F{_0}} = \{ (p_L,p_R) \in \R^2: H_L (p_L) = H_R (p_R) = F_{0} (p_L,p_R) \}.
\end{equation}

\begin{defi}[Maximal germs]\label{defi:max-germ}
  A germ $\cG$ is \emph{maximal} if any germ containing $\cG$ coincide with $\cG$.
\end{defi}

In order to define complete germs, we recall what the Riemann problem associated with $(v_-,v_+) \in \R^2$ is. It consists in solving \eqref{eq:scl} with the initial condition:

\begin{equation}\label{eq::C2}
v(0,x) = \begin{cases} v_- & \text{ if } x<0, \\ v_+ & \text{ if } x>0.\end{cases}
\end{equation}
Given $(v_-,v_+)\in \R^2$, 
a solution of the $\mathcal G$-Riemann problem with initial data \eqref{eq::C2} consists, for some suitable data $(p_L,p_R)\in \mathcal G$, in a standard Kruzhkov self-similar solution in $x<0$ joining $v_-$ at $t=0,x<0$ and $p_L$ at $t>0,x=0-$, with a jump  at $x=0,t>0$
from $p_L$ to $p_R$ and a standard Kruzhkov self-similar solution in $x>0$ joining $p_R$ at $x=0+,t>0$ and $v_+$ at $x>0,t=0$.

\begin{defi}[Complete germs]\label{defi:complete-germ}
  A germ $\cG$ is \emph{complete} if for all $(v_-,v_+) \in \R^2$, there exists a $\mathcal G$-entropy solution $v$ of \eqref{eq:scl} with initial data \eqref{eq::C2}.
\end{defi}

\begin{rem}\label{rem::C3}
In particular the traces $(p_L,p_R)$ at $x=0-,0+$ of the solution $v$ of the $\mathcal G$-Riemann problem lies in the germ $\mathcal G$. 
\end{rem}

\subsection{Entropy solutions}

\begin{defi}[Strong traces]\label{d:strong}
  Let $T>0$ and  $v \colon (0,T) \times \R \to \R$ be essentially bounded. We say that $v$ admits a \emph{strong trace} at $x=0-$ (resp. $x=0+$) if the function $x \mapsto v(\cdot,x)\in L^1((0,T))$
  has an essential limit in $L^1 ((0,T))$ as $x \to 0-$ (resp. $x \to 0+$).  
\end{defi}
E.~Yu.~Panov \cite{zbMATH05275060} proved that classical entropy solutions in $(0,T) \times (0,+\infty)$ (resp. $(0,T) \times (-\infty,0)$)
admit strong traces at $x=0+$ (resp. $x=0-$) as soon as there is no open interval on which flux functions $H_\alpha$ are constant. 
\begin{defi}[$\mathcal G$-entropy solutions -- \cite{MR2807133}]\label{defi:entropy}
  Let $\cG$ be a germ and $T>0$, let $v \colon (0,T) \times \R \to \R$ be essentially bounded and such that, for almost every $t \in (0,T)$, the function $v(t,\cdot)$ has strong traces at $x=0$, from
  the left $v(t,0-)$ and from the right $v(t,0+)$. It is a \emph{$\mathcal G$-entropy solution} of \eqref{eq:scl}-\eqref{eq:ic-scl} if
\begin{itemize}
\item  it is a classical entropy solution in $(0,T) \times (-\infty,0)$ and
  $(0,T) \times (0,+\infty)$;
\item for almost all $t \in (0,T)$, we have $(v(t,0-),v(t,0+)) \in \cG$;
\item $v(t,\cdot) \to v_0$ in $L^1_{\mathrm{loc}}(\R)$ as $t\to 0+$. 
\end{itemize}
\end{defi}
\begin{rem} \label{r:entropy-unique}
  $\mathcal G$-entropy solutions are unique (\cite[Theorem~3.11]{MR2807133}) as soon
  as the germ is \emph{definite}. In our case, this reduces to impose that the germ is \emph{maximal}
  since we only consider germs that are, following the terminology introduced in \cite{MR2807133}, $L^1$-dissipative.
\end{rem}

\subsection{Viscosity solutions}

In order to define viscosity solutions for the Hamilton-Jacobi equation, we have to specify the class of test functions we will work with. Following for instance \cite{MR3621434}, we use the following class. 
\begin{defi}[Test functions]
  A test function $\varphi \colon (0,T) \times \R \to \R$ is continuous and its restriction to $(0,T) \times [0,+\infty)$ and $(0,T) \times (-\infty,0]$ are continuously differentiable. For such a function $\varphi$ and $X_0= (t_0,0)$, $\partial_x^L \varphi(X_0)$ and $\partial_x^R \varphi(X_0)$ denote  derivatives in $x$ at $X_0$ of the restrictions of $\varphi$ to $(0,T) \times (-\infty,0]$ and $(0,T) \times [0,+\infty)$ respectively.
\end{defi}
Let $Q_T$ denote $(0,T) \times \R$ and $C^1_\wedge(Q_T)$ denote the set of test functions.

We also say that a test function $\varphi$ touches a function $u \colon Q_T \to \R$ from above (resp. from below) at $X_0$ if $u(X_0) = \varphi (X_0)$ and it there exists a neighbourhood $\mathcal{V}$ such that $u \le \varphi$ (resp. $u \ge \varphi$) in $\mathcal{V}$. 
\begin{defi}[$F_{0}$-viscosity solutions]\label{defi:viscosity}
Let $T>0$ and let $u \colon (0,T) \times \R$ be locally bounded. 
\begin{itemize}
\item The function $u$ is a \emph{sub-solution} of \eqref{eq:hj} if it is upper semi-continuous on $(0,T) \times \R$ and if, for all test function $\varphi \in C^1_{\wedge}(Q_T)$ touching $u$ from above at $X_0 =(t_0,x_0) \in (0,T) \times \R$, we have
  \begin{align*}
    \partial_t \varphi + H_L (\partial_x \varphi) \le 0 & \text{ at } X_0 \text{ if } x_0 < 0, \\
    \partial_t \varphi + H_R (\partial_x \varphi) \le 0 & \text{ at } X_0 \text{ if } x_0 > 0, \\
    \partial_t \varphi + \min (H_L(\partial_x^L \varphi), H_R(\partial_x^R \varphi),F_{0}(\partial_x^L \varphi, \partial_x^R \varphi))  \le 0 & \text{ at } X_0 \text{ if } x_0 =  0, 
  \end{align*}
\item The function $u$ is a \emph{super-solution} of \eqref{eq:hj} if it is lower semi-continuous on $(0,T) \times \R$ and if, for all test function $\varphi \in C^1_{\wedge}(Q_T)$ touching $u$ from below at $X_0 =(t_0,x_0) \in (0,T) \times \R$, we have
  \begin{align*}
    \partial_t \varphi + H_L (\partial_x \varphi) \ge 0 & \text{ at } X_0 \text{ if } x_0 < 0, \\
    \partial_t \varphi + H_R (\partial_x \varphi) \ge 0 & \text{ at } X_0 \text{ if } x_0 > 0, \\
    \partial_t \varphi + \max (H_L(\partial_x^L \varphi), H_R(\partial_x^R \varphi),F_{0}(\partial_x^L \varphi, \partial_x^R \varphi))  \ge 0 & \text{ at } X_0 \text{ if } x_0 =  0, 
  \end{align*}
\item The function $u$ is a \emph{solution} of \eqref{eq:hj} if its upper semi-continuous envelope $u^\ast$ is a sub-solution and its lower semi-continuous envelope $u_\ast$ is a super-solution. 
\end{itemize}
\end{defi}
\begin{rem}
  Since we will work with various nonlinearities $F_{0}$, it is convenient to simply say
that $u$ is an $F_{0}$-sub-solution of the Hamilton-Jacobi equation if
$u$ is a sub-solution of \eqref{eq:hj}. The same remark applies to super-solutions and solutions. 
\end{rem}

\begin{rem}\label{rem::C4}(Mapping the line onto the two half lines on the right)\\
Define
$$\left\{\begin{array}{lll}
\bar {u}^1(t,x):= {u}(t,x)& \quad \mbox{for}\quad x>0,&\quad H_1(p)=H_R(p)\\
\bar  {u}^2(t,x):= {u}(t,-x)& \quad \mbox{for}\quad x>0,&\quad H_2(p)=H_L(-p)\\
\bar F(p_1,p_2):=F(p_1,-p_2)&&
\end{array}\right.$$
Then HJ equation {\eqref{eq:hj}} is equivalent to
$$\left\{\begin{array}{lll}
\bar  {u}^i_t +H_i(\bar  {u}^i_x)&=0,&\quad x>0,\quad i=1,2\\
\bar  {u}^1(t,0)=\bar  {u}^2(t,0)&=:\bar  {u}(t)&\quad x=0\\
\bar  {u}_t +\bar F(\bar  {u}^1_x,\bar  {u}^2_x)&=0&\quad x=0\\
\end{array}\right.$$
which is the natural framework for HJ equations. This explains the signs that the reader may find strange  in the definitions below.
\end{rem}

\subsection{Characteristic points}

Characteristic points are first defined in the framework of HJ equations.
They are associated with the non-linearity $F_0$ at the origin. They are first
introduced in \cite{MR3621434} in the convex case and in \cite{MR3695961}  for non-convex Hamiltonians. See also  in \cite{forcadel2023nonconvex}.

\begin{defi}[Characteristic points for germs]\label{defi:charac-point-germ}
Let $\mathcal{G}$ be a germ for \eqref{eq:scl}.
\begin{itemize}
  \item
  A point $P = (p_L,p_R)$ lies in $\overline \chi(\mathcal{G})$ if { $P\in \mathcal G$ and if}
  there exists $\eps >0$ such that $H_\alpha(q_\alpha) > H_L(p_L) = H_R (p_R)$ for $\alpha \in \{L,R\}$ and $q_L \in (p_L-\eps,p_L)$ and $q_R \in (p_R,p_R+\eps)$. 
\item   A point $P = (p_L,p_R)$ lies in $\underline \chi (\mathcal{G})$ if { $P\in \mathcal G$ and if}
  there exists $\eps >0$ such that $H_\alpha(q_\alpha) < H_L(p_L) = H_R (p_R)$ for $\alpha \in \{L,R\}$ and $q_L \in (p_L,p_L+\eps)$ and $q_R \in (p_R-\eps,p_R)$. 
  \item We set $\chi(\mathcal G):=\underline{\chi}(\mathcal G)\cup \overline{\chi}(\mathcal G)$.
\end{itemize}
\end{defi}
\begin{defi}[Characteristic points for nonlinearities]\label{defi:charac-point-nonlin}
Let $F_{0}\colon \R \times \R$ be such that \eqref{e:assum-F} holds true. 
  A point $P = (p_L,p_R)$ lies in $\underline \chi (F_{0})$ (resp. $\overline \chi (F_{0})$) if $P \in \underline \chi (\cG_{F_{0}})$  (resp. $\overline \chi (\cG_{F_{0}})$) { where we recall that $\mathcal G_{F_0}$ is defined in \eqref{eq::C23}}. 
\end{defi}

\subsection{The relaxation operator}

We now define the relaxation operator $\cR$. It associates with any function $F_0$ satisfying
\eqref{e:assum-F} a new function $\mathcal{R} F_0$. In order to define it, it is convenient to define
\[
\underline{H}(p_L,p_R):=\max\left\{H_{L,+}(p_L),H_{R,-}(p_R)\right\}
\]
with
\[ H_{L,+}(p_L):=\inf_{q_L\ge p_L} H_L(q_L) \quad \text{ and } \quad H_{R,-}(p_R):=\inf_{q_R\le p_R}H_R(q_R).
\]
If $F_0\ge \underline{H}$, then it is possible to define the Godunov relaxation of $F_0$ as the map $F_0G:\R^2\to \R$ where
\begin{equation}
  \label{eq::C6}
(F_0G)(p_L,p_R)=\lambda \quad \mbox{s.t. there exists $(q_L,q_R)\in \R^2$ s.t.}\quad g^{H_L}(p_L,q_L)= F_0(q_L,q_R)=g^{H_R}(q_R,p_R)=:\lambda
\end{equation}
where $g^{H_\alpha}$ are Godunov fluxes associated to each flux $H_\alpha$ (see \eqref{eq::C16}).
It is important to notice that there may be several admissible values of $Q=(q_L,q_R)$, but it is possible to show that the value of $\lambda$ is uniquely defined.
Following  \cite{FM}, we define for any $F_0$ satisfying \eqref{e:assum-F}

\begin{equation}\label{eq::C1}
\mathcal RF_0:= (\max\left\{H_0,\underline{H}\right\})G.
\end{equation}

We recall some properties of the relaxation operator.
\begin{pro}[Properties of the relaxation operator, \cite{FM}]\label{p:prop-relax}
  Let $F_0$ be continuous, non-decreasing in the first variable et non-increasing in the second one. Then $\mathcal R F_0$ satisfies \eqref{e:assum-F} and we have
  \begin{enumerate}[(i)]
  \item $\cR F_0 =F_0 \quad \mbox{on}\quad \left\{F_0=H_{L}=H_{R}\right\}$\label{p:pr-i},
  \item $\cR^2=\cR$,
  \item $\cR F_0  \ge \underline H$
  \end{enumerate}
\end{pro}

\begin{rem}\label{rem::C8}
Notice that for $F_\varepsilon(p_L,p_R):=\varepsilon^{-1}(p_L-p_R)$, we have $\mathcal R F_\varepsilon\to F_0$ as $\varepsilon\to 0$, where
$$F_0(p_L,p_R):=g^{H_L}(p_L,z)=g^{H_R}(z,p_R)\quad \mbox{for some $z\in\R^2$}$$ 
which is exactly Diehl's condition (see   \cite[p.~28]{Diehl96}), obtained by vanishing viscosity. This shows that the natural relaxation operator  that we identified  in \eqref{eq::C6}, can be seen as a sort of generalization of Diehl's condition.
\end{rem}


\subsection{Maximality of germs associated to $\mathcal R F_0$}

\begin{lem}[Germs associated to $\mathcal R F_0$]\label{lem::C11}
Assume that $H^L,H^R$ satisfy \eqref{e:assum-H} and that $F_0$ satisfies \eqref{e:assum-F}.
Then the set $\mathcal G:=\mathcal G_{\mathcal R F_0}$ defined in \eqref{eq::C12} is a maximal  germ in the sense of Definitions \ref{defi:germ} and \ref{defi:max-germ}.
\end{lem}

\begin{proof}
The proof proceeds in two steps. 
  
\noindent {\sc Step 1: $\mathcal G$ is a germ.}
From Proposition \ref{p:prop-relax}, recall that $F:=\mathcal R F_0$ satisfies \eqref{e:assum-F}.
We then notice that by definition, the set $\mathcal G$ satisfies Rankine-Hugoniot relation, and that the monotonicity of $F$ implies the dissipation $D\ge 0$ on $\mathcal G\times \mathcal G$. Therefore $\mathcal G$ is a germ.\medskip

\noindent {\sc Step 2: $\mathcal G$ is maximal.}
Assume by contradiction that $\mathcal G$ is not maximal. Then there exists $\bar P=(\bar p_L,\bar p_R)\in \R^2\backslash \mathcal G$ such that 
\begin{equation}\label{eq::C15}
\hat {\mathcal G}:= \mathcal G \cup \left\{\bar P\right\}\quad \mbox{is a germ.}
\end{equation}
Hence
$$\bar \lambda:=H_L(\bar p_L)=H_R(\bar p_R) \not= F(\bar P)=:\lambda^*$$
From Proposition \ref{p:prop-relax}, recall that $F=\mathcal R F \ge \underline{H}$. Hence
by construction there exists $\bar Q=(\bar q_L,\bar q_R)\in\R^2$ such that
\[F(\bar p_L,{\bar p_R}){=\lambda^*}=g^{H_L}(\bar p_L,\bar q_L)=F(\bar q_L,\bar q_R)=g^{H_R}(\bar q_R,\bar p_R).\]

We first assume that  $\bar \lambda> \lambda^*$.
Then ${g^{H_R}(\bar p_R,\bar p_R)=}H_R(\bar p_R)=\bar \lambda> \lambda^* =F(\bar P)=g^{H_R}(\bar q_R,\bar p_R)${, which} implies that $\bar q_R< \bar p_R$. Similarly, we get $\bar  q_L>\bar p_L$. {Since $F(\bar p_L, \bar p_R)=F(\bar q_L, \bar q_R)$, }the monotonicities of $F$ imply that
$$F=const=\lambda^*\quad \mbox{on}\quad [\bar p_L,\bar q_L]\times [\bar q_R,\bar p_R]$$
In particular, the expressions of the Godunov fluxes $g^{H_\alpha}$ implies that there exists a unique 
$$Q^*:=(q_L^*,q_R^*)\in (\bar p_L,\bar q_L]\times [\bar q_R,\bar p_R)\quad \mbox{such that}\quad \left\{\begin{array}{l}
H_R> \lambda_*=H_R(q^*_R) \quad \mbox{on}\quad (q^*_R,\bar p_R]\\
H_L> \lambda_*=H_L(q^*_L) \quad \mbox{on}\quad [\bar p_L,q^*_L)\\
\end{array}\right.$$
Hence $Q^*\in \mathcal G \subset \hat {\mathcal G}$, and then
\begin{align*}
  0 & \le D(\bar P,Q^*)\\
  & =\mbox{sign}(\bar p_L-q^*_L)\cdot \left\{H_L(\bar p_L)-H_L(q^*_L)\right\}-\mbox{sign}(\bar p_R-q^*_R)\cdot \left\{H_R(\bar p_R)-H_R(q^*_R)\right\}\\
  & =-2(\bar \lambda-\lambda^*) \\
  & <0.
\end{align*}

Assume now that $\bar \lambda< \lambda^*$. We can argue as in the previous case and get a contradiction too.

Hence we conclude that \eqref{eq::C15} is false, and then $\mathcal G$ is maximal.
\end{proof}

\section{Criteria for viscosity and entropy solutions}
\label{s:criterion}

{We begin by defining the reduced set of test functions.}
\begin{defi}[Reduced set of test functions]\label{defi:reduction}
Let $F_{0}$ be such that \eqref{e:assum-F} holds true.  The  reduced set of test functions associated to $F_{0}$ for subsolutions (resp. supersolutions) is made of test functions $\varphi \in C^1_\wedge(Q_T)$ such that
$(\partial_x^L \varphi (t_0,0), \partial_x^R \varphi (t_0,0)) \in \underline\chi (F_{0})$ (resp. $(\partial_x^L \varphi (t_0,0), \partial_x^R \varphi (t_0,0)) \in \overline{\chi} (F_{0})$) for all $t_0 \in (0,T)$. It is denoted by $C^1_\wedge (Q_T,F_{0},SUB)$ (resp. $C^1_\wedge (Q_T,F_{0},SUP)$).
\end{defi}
{We now recall that for viscosity solutions, we have the following criterion.}
\begin{pro}[A criterion for $F_{0}$-sub- and $F_{0}$-super-solutions]\label{p:reduction}
  Let $u \colon Q_T \to \R$.
  \begin{itemize}
    \item If $u$ is upper-semi continuous, and $u$ is a (classical) viscosity sub-solution of \eqref{eq:hj} away from $x=0$, and for all $t_0 \in (0,T)$, 
\begin{equation}\label{eq::C9}
u (t_0,0) = \limsup_{(t,x) \to (t_0,0)} u (t,x),
\end{equation}
  and for all $\varphi \in C^1_\wedge(Q_T,F_{0},SUB)$ touching $u$ from above, we have
  \[ \partial_t \varphi + F_{0} (\partial_x^L \varphi, \partial_x^R \varphi) \le 0 \quad \text{ at } (t_0,0).\]
  Then $u$ is an $F_{0}$-sub-solution.
  \item If $u$ is lower-semi continuous and $u$ is a (classical) viscosity super-solution of \eqref{eq:hj} away from $x=0$, and for all $\varphi \in C^1_\wedge(Q_T,F_{0},SUP)$ touching $u$ from below, we have
  \[ \partial_t \varphi + F_{0} (\partial_x^L \varphi, \partial_x^R \varphi) \ge 0 \quad \text{ at } (t_0,0).\]
  Then $u$ is an $F_{0}$-super-solution.
\end{itemize}
\end{pro}

\begin{rem}\label{rem::C10}
Notice the dissymmetry between subsolutions and supersolutions in the result of Proposition  \ref{p:reduction}, with assumption \eqref{eq::C9} only for subsolutions. This comes from the fact that the Hamiltonians $H^L,H^R$ are assumed to be coercive (last line of  condition \eqref{e:assum-H}).
\end{rem}

\begin{pro}[Germs from characteristic points]\label{p:charac}
  Let $\mathcal{G}_F$ be a germ associated with a junction function $F= \mathcal{R} F_0$ and consider $P=(p_L,p_R) \in \R^2$
  satisfying $H_L(p_L) = H_R (p_R)$. If for all $Q \in \chi (\cG_F)$, we have $D(P,Q) \ge 0$, then $P \in \mathcal{G}_F$. 
\end{pro}
\begin{rem} \label{r:weaker}
  Later we will use  this result, by assuming more, namely assuming that $D(P,Q) \ge 0$ for all $Q \in \cG_F$. 
\end{rem}
\begin{proof}
We consider the function $u \colon (0,+\infty) \times \R$ defined by
\[ u (t,x) :=
  \begin{cases}
    - \lambda t + p_L x & \text{ if } x >0, \\
    - \lambda t + p_R x & \text{ if } x<0
  \end{cases}
\]
with $\lambda = H_L(p_L) = H_R (p_R)$. The function $u$ is a viscosity
solution of \eqref{eq:hj} in $(0,+\infty) \times \R \setminus \{ 0\}$.
We claim that the assumption ensures that it is also a viscosity solution
at $x=0$, i.e.
\[ \lambda = F (p_L,p_R).\]
This precisely means that $P \in \mathcal{G}_F$.

We are thus left with checking that $u$ is a viscosity solution at $x=0$.
We only prove that it is a sub-solution, the other case being similar.
It is enough to consider a test-function $\varphi$ of the form
\[ \varphi (t,x) = \Psi (t) +  \begin{cases}
                     q_L x & \text{ if } x <0, \\  q_R x & \text{ if } x >0
                   \end{cases}
\]
with $(q_L,q_R) \in \underline{\chi} (F)$. We assume that
$u \le \varphi$ in $(0,+\infty) \times \R$ and $u (t_0,0) = \varphi (t_0,0)$ for some $t_0>0$.
In particular, we have
\[ \Psi' (t_0) = -\lambda, \quad q_L \le p_L, \quad q_R \ge p_R.\]
The assumption of the proposition ensures that
\[ \sgn (p_L - q_L) (H_L (p_L) - H_L (q_L)) \ge  \sgn (p_R-q_R) (H_R (p_R) - H_R (q_R)).\]
Keeping in mind that  $H_L (p_L) = H_R (p_R)=\lambda$ and $H_L (q_L) = H_R (q_R) = F (q_L,q_R)$,
if $p_L > q_L$ or $p_R < q_R$, then we get
\[  (\lambda - F (q_L,q_R)) \ge 0.\]
The result is also  true if $p_L=q_L$ and $p_R=q_R$ since $\lambda = F(p_L,p_R)$.
Since $\Psi'(t_0) = - \lambda$, we finally get the desired viscosity inequality: $\Psi'(t_0) + F (q_L,q_R) \le 0$. 
\end{proof}

\section{The numerical scheme for the Hamilton-Jacobi equation}
\label{s:hj-scheme}

Before proving the convergence of the numerical scheme for the scalar conservation law,
we study the one associated with the Hamilton-Jacobi equation. It is necessary to study
it first, since we will use it to prove the convergence of the numerical scheme for the
scalar conservation law in the case where the nonlinearity $F_0$ is not necessarily
relaxed (\textit{i.e.} $F_0 \neq \cR F_0$). 

\subsection{Stability}

\begin{lem}[Stability of the numerical scheme]\label{l:stability}
  For all $t \ge 0$ and $x \in \R$, we have $|u_\Delta (t,x) -u_\Delta (0,x)| \le C_0 t$
  for
   $C_0= \max (C_L,C_R,C_{F_0})$ with $C_\alpha = \max_{|a| \le \|u_0\|_{\mathrm{Lip}}} |H_\alpha(a)|$ and
   $C_{F_0} = \max_{|a|,|b| \le \|u_0\|_{\mathrm{Lip}}}|F_0(a,b)|$.

   In particular, the function $u_\Delta$ is locally bounded in $L^\infty$, uniformly in $\Delta$. 
\end{lem}
\begin{proof}
  Since $u_\Delta$ is constant in time on intervals $[t_n,t_{n+1})$, it is enough to prove
  that $|u_\Delta (t_n, x) - u_\Delta (0,x) | \le C_0 t_n$. We prove it by induction on $n$.

  We only prove $u_\Delta (t_n,x) \le u_\Delta (0,x) + C_0 t_n$ since the other inequality
  can be proved in the same way. 
  It is true for $n=0$. We assume it is true for $n \ge 0$ and we prove it for $n+1$.
  In order to do so, we combine the induction assumption with the monotonicity of the scheme.
  Recalling the definition of $f_j$, see \eqref{e:def-fj}, we have 
   \begin{align*}
    u_j^{n+1} & = u_j^n + (\Delta t) f_j (v_{j-\frac12}^n, v_{j+\frac12}^n) \\
    & =: \mathcal H_j(u_{j-1}^n,u_j^n,u_{j+1}^n)\\
    & \le \mathcal H_j(u_{j-1}^0+C_0 t_n,u_j^0+C_0 t_n,u_{j+1}^0+C_0 t_n)\\
              & = (u^0_j + C_0 t_n) + (\Delta t) f_j (v^0_{j-\frac12},v^0_{j+\frac12}) \\
    & \le (u^0_j + C_0 n \Delta t) + C_0 (\Delta t). 
  \end{align*}
where in the third line, we have used the fact that the numerical scheme $\mathcal H_j$ for HJ equation is monotone under our CFL condition.
We conclude that $u_\Delta (t_{n+1},x) \le u_\Delta (0,x) + C_0 t_{n+1}$. 
\end{proof}

\subsection{Consistency}

The following lemma is very classical and straightforward. We skip the proof. 
\begin{lem}[Consistency of the numerical scheme]\label{l:consistency}
   Let \eqref{e:CFL} hold true. Let $(t,x) \in (0,T) \times \R$ and $\phi \in C^1_\wedge(Q_T)$ for some $T>0$.
  Assume that there exists $(t_\Delta,x_\Delta) = (n \Delta t,j \Delta x) \to (t,x)$ as $\Delta x \to 0$
  such that
  \begin{equation}\label{e:phi-eq}
    \frac{\phi (t_\Delta + \Delta t,x_\Delta) - \phi (t_\Delta,x_\Delta)}{\Delta t} 
    + f_{j} \left(\frac{\phi(t_\Delta,x_\Delta)-\phi (t_\Delta, x_\Delta-\Delta x)}{\Delta x},
    \frac{\phi(t_\Delta,x_\Delta+\Delta x)-\phi (t_\Delta, x_\Delta)}{\Delta x}\right) \le 0
\end{equation}
where $f_j$ is defined in \eqref{e:def-fj}.
\begin{itemize}
\item If $x >0$, then $\phi_t  + H_R (\phi_x ) \le 0$ at $(t,x)$. 
\item If $x<0$, then $\phi_t  + H_L (\phi_x ) \le 0$ at $(t,x)$. 
\item If $x=0$, then $\phi_t  + \min (H_R(\phi_x^R),H_L(\phi_x^L),F_0 (\phi_x^L,\phi_x^R) \le 0$ at $(t,x)$. 
\end{itemize}
\end{lem}

\subsection{Convergence of the scheme}

  We checked that the numerical scheme is monotone (thanks to the CFL condition~\eqref{e:CFL}),  stable (by  Lemma~\ref{l:stability}) and consistent (Lemma~\ref{l:consistency}). It is then known \cite{MR1115933} that it converges towards the unique
  solution of \eqref{eq:hj}-\eqref{eq:ic-hj}. Let us give some details for the reader's convenience. 
  \begin{proof}[Proof of Theorem~\ref{t:num-scheme-hj}]
    Let $u^+$ denote the upper relaxed limit of $u_\Delta$ as $\Delta x \to 0$ (recall \eqref{e:CFL}).
    Thanks to the stability of the scheme (Lemma~\ref{l:stability}), we know that $u^+$ is finite and
    $u^+(0,x) \le u_0 (x)$. Let us prove that it is a $F_{0}$-sub-solution of \eqref{eq:hj} in $(0,T) \times \R$ for all $T>0$.
    In order to do so, let $T>0$ and $\phi \in C^1_\wedge(Q_T)$ touching $u^+$ from above at $(t,x)$. We can assume without loss
    of generality that the contact is strict. We thus know that there exists $(t_\Delta, x_\Delta)$ such that
    \[u_\Delta -u_\Delta (t_\Delta,x_\Delta) \le \phi  - \phi (t_\Delta, x_\Delta).\]
    Let $\psi =   \phi + C_\Delta$ with $C_\Delta = u_\Delta (t_\Delta,x_\Delta) - \phi (t_\Delta, x_\Delta)$. 
    The monotonicity of the scheme implies that \eqref{e:phi-eq} holds true for $\psi$, and thus for $\phi$.
    Then Lemma~\ref{l:consistency} allows us to get the viscosity inequality.

    Analogously, we can prove that the lower relaxed limit $u_-$ of $u_\Delta$ as $\Delta x \to 0$ is a $F_{0}$-super-solution of \eqref{eq:hj}
    and $u_-(0,x) \ge u_0(x)$. The comparison principle (see \cite{FM}) then implies that $u_+ \le u_-$ and this implies that $u_\Delta$ converges
    locally uniformly towards $u$. 
\end{proof}

\section{The numerical scheme for the scalar conservation law}
\label{s:scl-scheme}

\subsection{Maximum principle}

It is classical that a monotone scheme enjoys a maximum principle. 
\begin{lem}[Maximum principle]\label{l:maximum}
We have: $\|v_\Delta \|_{L^\infty ((0,+\infty) \times \R)} \le \|v_0 \|_{L^\infty (\R)}$.
\end{lem}
\begin{proof}
  Let $M_0= \|v_0 \|_{L^\infty (\R)}$. The monotonicity of the scheme
  implies that $|v^n_j|\le M_0$ by arguing by induction on $n$. 
\end{proof}

\subsection{Discrete entropy inequalities and $L^1$-contraction}

An immediate consequence of the monotonicity of the scheme is the fact that the maximum of
two discrete sub-solutions is still a discrete sub-solution.
\begin{lem}[Maximum of discrete sub-solutions]\label{l:maximum-sub}
  Let $v^n_{j+\frac12}$ and $w^n_{j+\frac12}$ be such that
  \begin{align*}
    v^{n+1}_{j+\frac12} & \le \mathcal{F}_j (v^n_{j-\frac12}, v^n_{j+\frac12}, v^n_{j+\frac32}) \\
    w^{n+1}_{j+\frac12}  & \le \mathcal{F}_j (w^n_{j-\frac12}, w^n_{j+\frac12}, w^n_{j+\frac32}) 
  \end{align*}
  where $\mathcal{F}_j$ denotes the monotone scheme, see \eqref{e:monotone-scheme}. 
Then $V^n_{j+\frac12} = \max (v^n_{j+\frac12}, w^n_{j+\frac12})$ satisfies the same inequality. 
\end{lem}
Similarly, the minimum of discrete super-solutions is a discrete super-solutions.
Combining these two facts, we get the following discrete version of entropy inequalities (using $|a-b|=a\vee b - a\wedge b$).
\begin{lem}[Discrete entropy inequalities]\label{l:discrete-entropy}
  Let $v^n_{j+\frac12}$ and $w^n_{j+\frac12}$ be two solutions of the numerical scheme, 
 
  \begin{equation}\label{eq::C7}
\left\{\begin{array}{ll}
 v^{n+1}_{j+\frac12} & = \mathcal{F}_j (v^n_{j-\frac12}, v^n_{j+\frac12}, v^n_{j+\frac32}) \\
    w^{n+1}_{j+\frac12}  & = \mathcal{F}_j (w^n_{j-\frac12}, w^n_{j+\frac12}, w^n_{j+\frac32}) 
\end{array}\right.
\end{equation}
  where $\mathcal{F}_j$ denotes the monotone scheme, see \eqref{e:monotone-scheme}. 
  Then $V^n_{j+\frac12} = |v^n_{j+\frac12}- w^n_{j+\frac12}|$ satisfies
  \[\frac{V^{n+1}_{j+\frac12} - V^{n}_{j+\frac12}}{\Delta t} +     \frac{ \cQ_{j+1}^n - \cQ_{j}^n}{\Delta x}  \le 0 \]
  where
  \[ \cQ_{j}^n = Q_j (v^{n}_{j-\frac12},v^{n}_{j+\frac12};w^{n}_{j-\frac12},w^{n}_{j+\frac12}) \]
  with
  \( Q_j (a,b;c,d) = f_j (a \vee c, b \vee d) - f_j(a \wedge c, b \wedge d)\) and $f_j$ is given by formula~\eqref{e:def-fj}.
\end{lem}
We now state the $L^1$-contraction property of the scheme.
\begin{lem}[Discrete $L^1$-contraction]\label{l:discrete-contraction}
  Let $v^n_{j+\frac12}$ and $w^n_{j+\frac12}$ be two bounded solutions of the numerical scheme \eqref{eq::C7},
and assume that $|v_\Delta (0,\cdot) -w_\Delta (0,\cdot)|$ is integrable. 
  
Then for all $t = n \Delta t$ with $n \ge 1$, we have
\[ \int_{\R} |v_\Delta (t,x) - w_\Delta (t,x) | \dx \le \int_{\R} |v_\Delta (0,x) - w_\Delta (0,x) | \dx. \]
\end{lem}
\begin{proof}
  We first assume that $|v_\Delta (0,x) - w_\Delta (0,x)|$ has compact support. In this case,
  there exists $J \ge 1$ such that for all $s \in [0,t]$, $|v_\Delta(s,x) - w_\Delta (s,x)|=0$ for
  $|x| \ge x_J$. In particular, 
  \[ Q_j (v^{m}_{j-\frac12},v^{m}_{j+\frac12};w^{m}_{j -\frac12},w^{m}_{j +\frac12})=0 \]
  for $|j| \ge J+1$.

 We write
 the discrete entropy inequalities from Lemma~\ref{l:discrete-entropy} at time $m \in \{0,\dots,n-1\}$ and
 sum over $j \in \{-J-1,\dots, J+1\}$,
  \begin{equation}\label{e:1}
    \sum_{j \in -J-1}^{J+1} \frac{V^{m+1}_{j+\frac12} - V^{m}_{j+\frac12}}{\Delta t} \le 0.
  \end{equation}
  This yields the result if $|v_\Delta (0,x) - w_\Delta (0,x)|$ has compact support.

  We can now consider the sequence of numerical solutions associated with $v_\Delta (0,x)$ supported in the interval
  $[-N\Delta x,N\Delta x]$
  for $N \ge 1$. We can then easily pass to the limit in the inequality that we obtained in the first case.
  \end{proof}

\subsection{Continuous BV estimates}

In order to prove that the numerical solution associated with the scalar conservation law converges
towards the entropy solution, we derive discrete BV estimates in the time and the space variables. 
The computations at the discrete level follow their continuous counterpart closely. This is the
reason why we first explain how to derive BV estimates at the continuous level without justification.

Let $v$ be an entropy solution of \eqref{eq:scl}. 

\paragraph{Time BV estimate.}
Given a time increment $h$, the function $w(t,x) = v(t+h,x)$ is also an entropy solution of \eqref{eq:scl}
and the $L^1$-contraction property (with finite speed of propagation) implies that
\[ \int_{\R} |v(t+h,x) - v(t,x)| \dx \le \int_{\R} |v(h,x) - v(0,x)| \dx.\]
Dividing by $h$ and letting $h \to 0$, we obtain
\[ \int_{\R} |\partial_t v(t,x)| \dx \le \int_{\R} |\partial_t v(0,x)| \dx.\]
We now use that $\partial_t v = - \partial_x (H_\alpha(v)) = - H'_\alpha (v) v_x$
and in particular,
\[|\partial_t v(0,x)|\le L |(v_0)_x|.\]
We conclude that 
\[ \int_{\R} |\partial_t v(t,x)| \dx \le L \|v_0\|_{\BV}.\]

\paragraph{Space BV estimate.}
Given a spatial increment $h$, since $v$ is an entropy solution of \eqref{eq:scl}, we know that
it satisfies for $t>0$ and $x>0$,
\[ \partial_t |v(t,x+h) - v(t,x)| + \partial_x q_R (v(t,x),v(t,x+h) \le 0.\]
Integrating this inequality on $[t_1,t_2] \times [a_R,b_R]$ with $0<t_1<t_2$ and $0<a_R < b_R$, we get
\begin{align*}
  \int_{[a_R,b_R]} |v(t_2,x+h)-v(t_2,x)| \dx &\le \int_{[a_R,b_R]} |v(t_1,x+h)-v(t_1,x)| \dx \\
  & + \sum_{c=a_R,b_R} \int_{t_1}^{t_2} |q_R (v(t,c),v(t,c+h))| \dt .
\end{align*}
We now estimate the right hand side of the previous inequality.
Recalling the definition of $q_R$, see \eqref{e:qalpha}, we have 
\begin{align*}
  \int_{t_1}^{t_2} |q_R (v(t,c),v(t,c+h))|  \dt &\le   \int_{t_1}^{t_2} |H_R (v(t,c)) - H_R (v(t,c+h))|  \dt \\
                                                &\le   \int_{t_1}^{t_2} \left| \int_{[c,c+h]} \partial_t v(t,x) \dx \right|  \dt \\
  & \le \int_{[c,c+h]} \| v(\cdot,x)\|_{\BV ([t_1,t_2])} \dx. 
\end{align*}
We thus get,
\[
  \int_{[a_R,b_R]} |v(t_2,x+h)-v(t_2,x)| \dx \le \int_{[a_R,b_R]} |v(t_1,x+h)-v(t_1,x)| \dx 
   + \sum_{c=a_R,b_R} \int_{[c,c+h]} \| v(\cdot,x)\|_{\BV ([t_1,t_2])} \dx.
\]
Dividing by $h\to 0$, we formally get
$$\int_{[a_R,b_R]} |\partial_x v(t_2,x)|\ dx \le \int_{[a_R,b_R]} |\partial_x v(t_1,x)|\ dx +\sum_{c=a_R,b_R}  ||v(\cdot,c)||_{BV([t_1,t_2])}$$

\subsection{Discrete BV estimates}

We first show that discrete solutions of the scalar conservation law have bounded variation in the time variable.
This is a classical consequence of the $L^1$-contraction property.
\begin{lem}[Discrete time BV estimate]\label{l:t-BV-discrete}
  Let $v^n_{j+\frac12}$ be a solution of the discrete numerical scheme such that \eqref{e:v0-discrete} holds true for
  some $v_0 \in \BV (\R)$. Then 
  \[ \int_{\R} \frac{\left|v_\Delta (t+\Delta t,x)-v_\Delta (t,x)\right|}{\Delta t} \dx \le 2 \mathfrak{L} \|v_0\|_{\BV(\R)} \]
  where $\mathfrak{L} = \max_{\alpha \in \{0,L,R\}} \mathfrak{L}_\alpha$. 
\end{lem}
\begin{proof}
We first  use the formula for $v^1_{j+\frac12}$ and the fact that $f_j$ is $\mathfrak{L}$-Lipschitz continuous in order to get,
  \[ \frac{\left|v^1_{j+\frac12}-v^0_{j+\frac12} \right|}{\Delta t} \le \frac{\mathfrak{L}}{\Delta x} \left(  |v^0_{j+\frac12} - v^0_{j-\frac12}| + |v^0_{j+\frac32} - v^0_{j+\frac12}|\right). \]
  We now use the formula for $v^0_j$, see \eqref{e:v0-discrete}, and we get,
  \[ |v^0_{j+\frac32} - v^0_{j+\frac12}| \le \frac{1}{\Delta x} \int_{x_j}^{x_{j+1}} |v_0 (y + \Delta x) - v_0(y)| \, \mathrm{d} y. \]
This shows that $v_\Delta (\Delta t,x) - v_\Delta (0,x)$ is integrable. 

Then the discrete $L^1$-contraction (Lemma~\ref{l:discrete-contraction})  applied to $v^n_{j+\frac12}$ and $w^n_{j+\frac12} = v^n_{j+1+\frac12}$ yields the result for all $n \ge 0$, using the fact that $\displaystyle \int_\R \left|\frac{v_\Delta(x+\Delta x)-v_0(x)}{\Delta x}\right|\ dx \le ||v_0||_{BV(\R)}$.
\end{proof}

We now turn to BV estimates in the spatial variable. Such estimates are less classical but known, see for instance \cite[Lemma~4.2]{MR2379885}.
We provide an alternative proof following the reasoning at the continuous level presented above. 
\begin{lem}[Discrete space BV estimate]\label{l:x-BV-discrete}
  Let $v^n_{j+\frac12}$ be a solution of the discrete numerical scheme. 
  Let $j_R,J_R \ge 1$ and $j_L,J_L \le -2$ with $J_L \le j_L$ and $j_R \le J_R$.
  Then for any integers $n_1,n_2$ such that $1 \le n_1 < n_2$, we have,
  \begin{align*}
    \sum_{j=j_R}^{J_R} \frac{|v^{n_2}_{j+\frac12+1}-v^{n_2}_{j+\frac12}|}{\Delta x}
     \le \sum_{j=j_R}^{J_R} \frac{|v^{n_1}_{j+\frac12+1}-v^{n_1}_{j+\frac12}|}{\Delta x}
    + \Ccfl \sum_{n=n_1}^{n_2-1} \frac{\left| v^{n+1}_{j_R+\frac12}- v^n_{j_R+\frac12} \right|}{\Delta t}
    + \Ccfl \sum_{n=n_1}^{n_2-1}  \frac{\left| v^{n+1}_{J_R+1+\frac12}- v^n_{J_R+1+\frac12} \right|}{\Delta t},\\
        \sum_{j=J_L}^{j_L} \frac{|v^{n_2}_{j+\frac12+1}-v^{n_2}_{j+\frac12}|}{\Delta x}
     \le \sum_{j=J_L}^{j_l} \frac{|v^{n_1}_{j+\frac12+1}-v^{n_1}_{j+\frac12}|}{\Delta x}
    + \Ccfl \sum_{n=n_1}^{n_2-1} \frac{\left| v^{n+1}_{J_L+\frac12}- v^n_{J_L+\frac12} \right|}{\Delta t}
    + \Ccfl \sum_{n=n_1}^{n_2-1}  \frac{\left| v^{n+1}_{j_L+1+\frac12}- v^n_{j_L+1+\frac12} \right|}{\Delta t}.
  \end{align*}
\end{lem}
\begin{rem}\label{r:mean-x}
  In order to use the time BV estimate from Lemma~\ref{l:t-BV-discrete}, it is necessary to consider a mean (i.e. to integrate) in
  the $x$ variable in the right hand side, that is to say in $j_R$ and $J_R$.
\end{rem}
\begin{proof}
  We only do the proof at the right hand side of the origin since the estimate on the other side is identical. 
Considering $w^n_{j+\frac12} = v^n_{j+1+\frac12}$ and integrating in the discrete variables $n$ and $j$
the estimate from Lemma~\ref{l:discrete-entropy} yields the following discrete BV estimate away from $x=0$.  
  \begin{multline*}
    \sum_{j=j_R}^{J_R} |v^{n_2}_{j+\frac12+1}-v^{n_2}_{j+\frac12}| 
     \le \sum_{j=j_R}^{J_R} |v^{n_1}_{j+\frac12+1}-v^{n_1}_{j+\frac12}|  \\
    +  \sum_{n=n_1}^{n_2-1} Q^R (v^n_{j_R-\frac12},v^n_{j_R+\frac12}; v^n_{j_R+\frac12}, v^n_{j_R+\frac32}) \frac{\Delta t}{\Delta x} 
    -  \sum_{n=n_1}^{n_2-1} Q^R (v^n_{J_R+1-\frac12},v^n_{J_R+1+\frac12}; v^n_{J_R+1+\frac12}, v^n_{J_R+1+\frac32}) \frac{\Delta t}{\Delta x}.
  \end{multline*}
  where  \( Q^R (a,b;c,d) = g^{H_R} (a \vee c, b \vee d) - g^{H_R}(a \wedge c, b \wedge d).\)
Thanks to the technical lemma~\ref{l:tech}, we get,
  \begin{align*}
    \sum_{j=j_R}^{J_R} |v^{n_2}_{j+\frac12+1}-v^{n_2}_{j+\frac12}| 
     \le& \sum_{j=j_R}^{J_R} |v^{n_1}_{j+\frac12+1}-v^{n_1}_{j+\frac12}|  \\
    &+  \sum_{n=n_1}^{n_2-1} \left|g^{H_R} (v^n_{j_R-\frac12},v^n_{j_R+\frac12}) - g^{H_R}(v^n_{j_R+\frac12}, v^n_{j_R+\frac32}) \right| \frac{\Delta t}{\Delta x} 
   \\
    &+  \sum_{n=n_1}^{n_2-1} \left|g^{H_R} (v^n_{J_R+1-\frac12},v^n_{J_R+1+\frac12})  - g^{H_R}(v^n_{J_R+1+\frac12}, v^n_{J_R+1+\frac32}) \right| \frac{\Delta t}{\Delta x}.
  \end{align*}
Recalling the definition of the scheme, see \eqref{eq:scheme-SCL}, and the CFL condition, see \eqref{e:CFL}, we obtain the desired estimate. 
\end{proof}

We used the following technical lemma in the proof of the spatial BV estimates.
It can be viewed as the discrete counterpart of the elementary inequality $|q_\alpha (a,c)| \le |H_\alpha (a) - H_\alpha (c)|$ at the continuous level. Recall that $q_\alpha(\cdot,b)$ is the flux function associated
with the entropy function $|\cdot -b|$, see \eqref{e:qalpha}.
\begin{lem}\label{l:tech}
For all $a,b,c \in \R$, we have: $|Q^\alpha (a,b;b,c) | \le \left|g^{H_\alpha} (a,b) - g^{H_\alpha} (b,c) \right|$.
\end{lem}
\begin{proof}
  We want to study 
  \[ D_Q := Q^\alpha (a,b;b,c) = g^{H_\alpha} (a \vee b, b \vee c) - g^{H_\alpha}(a \wedge b, b \wedge c). \]
  We distinguish cases by examining the values taken by $g^{H_\alpha} (a \vee b, b \vee c)$.

  If $g^{H_\alpha} (a \vee b, b \vee c) = g^{H_\alpha} (a,b)$, then $a \ge b$ and $c \le b$. In particular,
  $g^{H_\alpha}(a \wedge b, b \wedge c) = g^{H_\alpha}(b,  c)$ and we get the desired estimate.

  If $g^{H_\alpha} (a \vee b, b \vee c) = g^{H_\alpha} (a,c)$, then $b \le a$ and $ b\le c$. In particular,
  $g^{H_\alpha}(a \wedge b, b \wedge c) = g^{H_\alpha}(b,b)$. In this case, we have
  \begin{align*}
    D_Q = g^{H_\alpha} (a,c)- g^{H_\alpha}(b,b) \le  g^{H_\alpha} (a,b)- g^{H_\alpha}(b,c), \\
    D_Q = g^{H_\alpha} (a,c)- g^{H_\alpha}(b,b) \ge  g^{H_\alpha} (b,c)- g^{H_\alpha}(a,b).
  \end{align*}
  These inequalities also imply the desired result.

  If $g^{H_\alpha} (a \vee b, b \vee c) = g^{H_\alpha} (b,c)$, then $a \le b$ and $ b\le c$.
  In particular, $g^{H_\alpha}(a \wedge b, b \wedge c) = g^{H_\alpha}(a,b)$ and
  $D_Q = g^{H_\alpha} (b,c) - g^{H_\alpha}(a,b)$ and we conclude in this case too.

  Finally if $g^{H_\alpha} (a \vee b, b \vee c) = g^{H_\alpha} (b,b)$, then $a \le b$ and $ c \le b$.
  In particular, $g^{H_\alpha}(a \wedge b, b \wedge c) = g^{H_\alpha}(a,c)$ and we have
  \begin{align*}
    D_Q = g^{H_\alpha} (b,b) - g^{H_\alpha}(a,c)  \le g^{H_\alpha} (b,c) - g^{H_\alpha} (a,b), \\
    D_Q = g^{H_\alpha} (b,b) - g^{H_\alpha}(a,c)  \ge g^{H_\alpha} (a,b) - g^{H_\alpha} (b,c). 
  \end{align*}
  The proof of the lemma is now complete. 
\end{proof}

\subsection{Proof of  convergence}

We prove simultaneously Theorems~\ref{t:num-scheme-scl} and \ref{t:link}.
\begin{proof}[Proof of Theorems~\ref{t:num-scheme-scl} and \ref{t:link}]
We first consider any $F_0$.

  In order to prove that $v_\Delta$ converges towards the entropy solution $v$ of \eqref{eq:scl} in $L^1$ locally in time and space,
  we use the maximum principle (Lemma~\ref{l:maximum}) and the discrete BV estimates in time and space from Lemmas~\ref{l:t-BV-discrete} and \ref{l:x-BV-discrete} (see Remark~\ref{r:mean-x}). These estimates give
  $$\forall \delta \in(0,1),\quad \forall T>0,\quad |v_\Delta|_{BV(\Omega_{\delta,T})} \le C_{\delta,T}\quad \mbox{with}\quad \Omega_{\delta,T}:=((-\delta^{-1},-\delta)\cup (\delta,\delta^{-1}))\times (0,T)$$
  where the constant $C_{\delta,T}$ is independent on $\Delta$ small.
  Because $v_\Delta$ is also bounded in $L^\infty$, this implies that $v_\Delta$ is compact in $L^1(K)$ for any compact set $K \subset [0,+\infty) \times \R$ (see for instance \cite[Theorem~5.5]{MR3409135}). Consequently, we can extract a  subsequence (still denoted by $\Delta$) such that $v_\Delta$ converges in $L^1_{\mathrm{loc}}$ and almost everywhere as $\Delta \to 0$. We are going to prove that the limit $v$ is the unique entropy solution of \eqref{eq:scl} submitted to the initial condition \eqref{eq:ic-scl}.

\paragraph{\sc Deriving the entropy inequalities away from the origin.}
  Let $\kappa \in \R$.  Using Lemma~\ref{l:discrete-entropy} with $w^n_{j+\frac12} = \kappa$, we know that we have for all $x > \Delta x$, 
 \begin{multline*}
   \frac{|v_\Delta (t+\Delta t,x)-\kappa| - |v_\Delta (t,x)-\kappa|}{\Delta t} \\
   + \frac1{\Delta x} \left( Q^R (v_\Delta (t,x),v_\Delta (t,x+\Delta x);\kappa,\kappa ) - Q^R (v_\Delta (t,x-\Delta x),v_\Delta (t,x);\kappa,\kappa ) \right) \le 0 
 \end{multline*}
 where we recall that
  \[ Q^R (a,b;\kappa,\kappa) = g^{H_R} (a \vee \kappa, b \vee \kappa) - g^{H_R} (a \wedge \kappa, b \wedge \kappa).\]
  In particular,
  \[ Q^R (v_\Delta (t,x-\Delta x),v_\Delta (t,x);\kappa,\kappa )\to H_R (v(t,x) \vee \kappa) - H_R (v(t,x) \wedge \kappa) = q_R (v(t,x),\kappa)\]
  almost everywhere. Integrating against a non-negative test function $\phi \in C^\infty_c([0,+\infty)\times (0,+\infty))$ and using the dominated convergence theorem then implies that for all $\kappa \in \R$, 
  \[ \int_{\R^2} (|v- \kappa| \partial_t \phi  + q_R(v,\kappa) \partial_x \phi) \dt \dx
    + \int_{\R} \phi (0,x) |v_0 (x)-\kappa|\dx \ge 0.\]
Similarly, we have for all $\kappa \in \R$ and all non-negative test function $\phi \in C^\infty_c([0,+\infty)\times (-\infty,0))$,
\[\int_{\R^2} (|v- \kappa| \partial_t \phi  + q_R(v,\kappa) \partial_x \phi) \dt \dx
    + \int_{\R} \phi (0,x)  |v_0 (x)-\kappa| \dx \ge 0.\]
    Then it is classical (using the finite speed of propagation and localized $L^1$ contraction) that this implies that the essential limit of $v(t,\cdot)$ is $v_0$ in $L^1_{\mathrm{loc}}(0,+\infty)$ as $t\to 0+$. The same argument for the left side allows us to conclude that $v$ has  initial data $v_0$ in the sense of Definition \ref{defi:entropy}.

\paragraph{\sc Weak formulation in the special case $F_0=\mathcal R F_0$.}  Since $v$ is an entropy solution of a scalar conservation law away from $x=0$, we can apply Panov's theorem \cite{zbMATH05275060}
  and deduce that $v$ admits strong traces on both sides (see Definition~\ref{d:strong}). 

  If $v(t,0\pm)$ denotes the strong traces of $v(t,\cdot)$ at the discontinuity $x=0$, we are left with checking  that  $P = (v(t,0-),v(t,0+))$
  is in the germ $\cG = \cG_{\cR F_0}$. It is convenient to write $P=(p_L,p_R)$. 

  In order to do so, we prove that
  \begin{equation}\label{e:CS}
    \forall K \in \cG_{F_0}, \qquad D(P,K) \ge 0.
  \end{equation}
  Let $K=(\kappa_L,\kappa_R)$, let $\phi \in C^\infty_c ((0,+\infty) \times \R)$ be non-negative and let 
  \[ w_\Delta (t,x) = \kappa_\Delta(x) = \begin{cases} \kappa_L &\text{ if } x<0, \\ \kappa_R &\text{ if } x>0. \end{cases} \]
  It is a solution of the numerical scheme for the scalar conservation law if and only if $K \in \cG_{F_0}$. 
  
  We now integrate the discrete entropy inequality from Lemma~\ref{l:discrete-entropy} with $w_\Delta = \kappa_\Delta$ against $\phi$, 
  \begin{multline*}
    \int_{(0,+\infty) \times \R} \frac{|v_\Delta-\kappa_\Delta| (t+\Delta t, x) -|v_\Delta-\kappa_\Delta| (t,x)}{\Delta t} \phi (t,x) \dt \dx
    \\+ \int_{(0,+\infty) \times \R} \frac{\cQ_\Delta (t, x+\Delta x) -\cQ_\Delta (t,x)}{\Delta x} \phi (t,x) \dt \dx \le 0
  \end{multline*}
  where
  \[ \cQ_\Delta (t,x)  =  \cQ_{j}^m \text{ for } (t,x) \in [t_m, t_{m+1}) \times [x_j, x_{j+1})\]
with $\cQ_{j}^m$ defined in the statement of Lemma~\ref{l:discrete-entropy}. 

For $\Delta x$ (and then $\Delta t$) small enough so that $\phi$ is supported    in $[t_1,+\infty)$, we get after integrating by parts,
  \[ \int_{(0,+\infty) \times \R} |v_\Delta-\kappa_\Delta| (t,x) \frac{\phi (t, x) -\phi (t-\Delta t,x)}{\Delta t}  \dt \dx
    + \int_{(0,+\infty) \times \R} \cQ_\Delta (t,x) \frac{\phi (t, x) -\phi (t,x-\Delta x)}{\Delta x}  \dt \dx \ge 0.\]
  We examine the function $\cQ_\Delta (t,x)$. We have,
  \[ \cQ_\Delta (t,x) =
    \begin{cases}
      Q^R (v_\Delta (t,x-\Delta x),v_\Delta (t,x);\kappa_R,\kappa_R) & \text{ if } x>  \Delta x, \\
      Q^L (v_\Delta (t,x-\Delta x), v_\Delta (t,x);\kappa_L,\kappa_L) & \text{ if } x < 0, \\
      Q_0 (v_\Delta (t,x-\Delta x) ,v_\Delta (t,x); \kappa_L,  \kappa_R)  & \text{ if } 0 < x < \Delta x.
    \end{cases}
  \]
We thus can write 
\[ \int_{(0,+\infty) \times \R} \cQ_\Delta (t,x) \left( \frac{\phi (t, x) -\phi (t,x-\Delta x)}{\Delta x} \right) \dt \dx
  = \mathcal{D}_R + \mathcal{D}_L + \mathcal{D}_0\]
with
\begin{align*}
  \mathcal{D}_R & = \int_{(0,+\infty) \times [{\Delta x},+\infty)} Q^R (v_\Delta (t,x-\Delta x),v_\Delta (t,x);\kappa_R,\kappa_R) \left(\frac{\phi (t, x) -\phi (t,x-\Delta x)}{\Delta x} \right) \dt \dx\\
    \mathcal{D}_L & = \int_{(0,+\infty) \times (-\infty,0]} Q^L (v_\Delta (t,x-\Delta x),v_\Delta (t,x);\kappa_L,\kappa_L) \left( \frac{\phi (t, x) -\phi (t,x-\Delta x)}{\Delta x}  \right)\dt \dx\\
  \mathcal{D}_0 & = \int_{(0,+\infty) \times [0,\Delta x]}  Q^0 (v_\Delta (t,x-\Delta x),v_\Delta (t,x);\kappa_L,\kappa_L) \left( \frac{\phi (t, x) -\phi (t,x-\Delta x)}{\Delta x}  \right)\dt \dx.
\end{align*}
We now pass to the limit in the resulting inequality,
\[ \int_{(0,+\infty) \times \R} |v_\Delta-\kappa_\Delta| (t,x) \left( \frac{\phi (t, x) -\phi (t-\Delta t,x)}{\Delta t}\right)  \dt \dx
    + \mathcal{D}_R + \mathcal{D}_L + \mathcal{D}_0 \ge 0.\]
  It is easy to pass to the limit in the first three terms. As far as $\mathcal{D}_0$ is concerned, it goes to $0$ as $\Delta x \to 0$. We finally get,
\[\int_{(0,+\infty) \times (0,+\infty)} (|v-\kappa_R| \phi_t + q_R(v,\kappa_R) \phi_x) \dt \dx 
     + \int_{(0,+\infty) \times (-\infty,0)} (|v-\kappa_L| \phi_t + q_L (v,\kappa_L) \phi_x)   \dt \dx  \ge 0.\]
Now choosing a test function of the form $\phi(t,x)=\phi_\varepsilon(t,x)=\psi(t)\cdot  \max\left\{0,1-\varepsilon^{-1}|x|\right\}$, which focuses on the interface $x=0$ as  $\varepsilon\to 0$, we get a boundary term, which is well defined from the existence of strong traces. This gives an  inequality for all $0\le \psi \in C^1_c(0,T)$, which implies that 
\[ D(P,K) \ge 0 .\]
We thus proved \eqref{e:CS} and we can apply Proposition~\ref{p:charac} (recall that $F_0 = \cR F_0$ and see Remark~\ref{r:weaker}) and obtain that $P  = (v(t,0-),v(t,0+)) \in \mathcal G_{F_0}$. Therefore $v$ is a $\mathcal G_{F_0}$-entropy solution \eqref{eq:scl} with  initial data $v_0$. The uniqueness of $v$ follows from the maximality of the germ $\mathcal G_{F_0}$ (see Lemma \ref{lem::C11}).

\bigskip

\paragraph{\sc General case when $F_0\not\equiv \mathcal R F_0$.}

We now treat the general case, that is to say we do not assume anymore that $F_0 = \cR F_0$.
In this case, we consider the numerical solution $\bar u_\Delta$  for the HJ equation associated with $\cR F_0$ and
 the numerical solution $\bar v_\Delta$ of the conservation law associated with $\cR F_0$. We have in particular $\bar v_\Delta=\partial_x \bar u_\Delta$.
 We know that $\bar u_\Delta$ converges towards
the unique $\cR F_0$-viscosity solution $\bar u$ of \eqref{eq:hj}-\eqref{eq:ic-hj}  and $\bar v_\Delta$
converges towards the unique $\cG_{\cR F_0}$-entropy solution $\bar v$ of \eqref{eq:scl},\eqref{eq:ic-scl} and
$(\bar u)_x = \bar v$.
We now also consider $u_\Delta$ and $v_\Delta$ the numerical schemes associated with $F_0$. We know that 
$u_\Delta$ converges towards the unique $F_0$-viscosity solution $u$ of \eqref{eq:hj},\eqref{eq:ic-hj},
which is also the unique $\mathcal R F_0$-viscosity solution (by \cite{FM}). Hence $u = \bar u$. Moreover, $v_\Delta$ converges in $L^1_{\mathrm{loc}}([0,T)\times \R)$ towards $u_x = (\bar u)_x = \bar v$. We thus proved
that $v_\Delta$ converges towards the unique $\cG_{\cR F_0}$-entropy solution of \eqref{eq:scl},\eqref{eq:ic-scl}. 
\end{proof}

\section{Classification of maximal complete germs: proof of Theorem \ref{t:germ}}\label{s:germs}

We start with the following independent result (whose proof is quite long and then can not be reproduced here)
\begin{theo}[Complete germs are classified, \cite{Mon}]\label{th::C22}
Assume \eqref{e:assum-H}. Let $\mathcal G\subset \R^2$ be a complete germ in the sense of Definitions \ref{defi:germ} and \ref{defi:complete-germ}. Then $\mathcal G$ is maximal in the sense of Definition \ref{defi:max-germ}. Moreover, there exists a function $F$ satisfying \eqref{e:assum-F} such that $F=\mathcal R F$ and $\mathcal G=\mathcal G_F$ with $\mathcal G_F$ defined in \eqref{eq::C23}.
\end{theo}

\begin{proof}[Proof of Theorem \ref{t:germ}]
  We set $F:=\mathcal R F_0$.
  
\noindent {\sc Step 1: $\mathcal G_{F}$ is a maximal and complete germ.}
From Lemma \ref{lem::C11}, we already know hat $\mathcal G_{F}$ is a maximal germ. Now Theorem \ref{t:num-scheme-scl} shows the existence of a $\mathcal G_F$-entropy solution for any suitable initial data, including the ones for the Riemann problem. This shows the completeness of the germ $\mathcal G_F$.

\noindent {\sc Step 2: identification of any (maximal) complete germ $\mathcal G$.}
If $\mathcal G$ is a complete germ, then $\mathcal G=\mathcal G_F$ with $F=\mathcal R F$ follows from Theorem \ref{th::C22}.
\end{proof}

\end{document}